\documentclass[10pt,a4paper,draft]{amsart}
\usepackage[latin1]{inputenc}
\usepackage{amsmath, verbatim}
\usepackage{amsopn}
\usepackage{amsfonts}
\usepackage{amssymb}
\usepackage{graphicx}

\let\span\relax
\DeclareMathOperator{\span}{span}

\begin{document}

\newcommand{\sgn}{\operatorname{sgn}}

\newtheorem{theorem}{Theorem}[section]
\newtheorem{lemma}[theorem]{Lemma}
\newtheorem{proposition}[theorem]{Proposition}
\newtheorem{corollary}[theorem]{Corollary}
\newtheorem{definition}[theorem]{Definition}
\newtheorem{example}[theorem]{Example}
\newtheorem{remark}[theorem]{Remark}

\def\a{\alpha}
\def\b{\beta}
\def\d{\delta}
\def\g{\gamma}
\def\l{\lambda}
\def\o{\omega}
\def\s{\sigma}
\def\t{\tau}
\def\r{\rho}
\def\D{\Delta}
\def\G{\Gamma}
\def\O{\Omega}
\def\e{\epsilon}
\def\p{\phi}
\def\P{\Phi}
\def\S{\Psi}
\def\E{\eta}
\def\m{\mu}
\def\Lam{\Lambda}
\def\eff{\mathcal{F}}
\def\ess{\mathfrak{P}}
\def\eee{\mathcal{E}}
\newcommand{\reals}{\mathbb{R}}
\newcommand{\naturals}{\mathbb{N}}
\newcommand{\ints}{\mathbb{Z}}
\newcommand{\complex}{\mathbb{C}}
\newcommand{\rationals}{\mathbb{Q}}
\newcommand{\downss}{\rotatebox{-90}{$\subseteq$}}
\newcommand{\nubar}{\overline{\nu}}
\newcommand{\ubar}{\overline{u}}
\newcommand{\obar}{\overline{\o}}

\author{Travis B. Russell}
\title{Characterizations of ordered operator spaces}
\address{Department of Mathematics, University of Nebraska-Lincoln, Lincoln NE 68588-0130}
 \email{trussell8@huskers.unl.edu}
\urladdr{http://www.math.unl.edu/~trussell8/}

\date{September 9, 2016. Revised February 15, 2017.}
 \maketitle

\begin{abstract} We demonstrate new abstract characterizations for
unital and non-unital operator spaces. We characterize unital
operator spaces in terms of the cone of accretive operators
(operators whose real part is positive). We show that matrix norms
and accretive cones are induced by gauges, although inducing gauges
are not unique in general. Finally, we show that completely
positive completely contractive linear maps on non-unital operator
spaces extend to any containing operator system if and only if the
operator space is induced by a unique gauge.
\end{abstract}

\section{Introduction}

The study of operator spaces and operator systems has played an
increasingly important role in operator theory and operator algebras
since the introduction of completely positive maps by Stinespring in
\cite{Stinespring} and the seminal work on completely positive maps
by Arveson in \cite{Arveson69}. \textbf{Operator spaces} (vector
subspaces of $B(H)$, the bounded linear operators on a Hilbert space
$H$) have been studied in the context of Ruan's Theorem (Theorem
\ref{Ruan} below), which characterizes them up to complete isometry.
\textbf{Operator systems} (unital self-adjoint operator spaces), on
the other hand, have been studied in the context of the Choi-Effros
Theorem (Theorem \ref{ChoiEffros} below) which characterizes them
 up to unital complete order isomorphism. Since the norm
structure of an operator system is determined by its
order structure, operator systems are examples of ordered operator
spaces - operator spaces possessing a specified cone of positive operators at
each matrix level. Abstract operator spaces,
on the other hand, lack natural cones of positive operators.
For example, if $V \subset B(H)$ is a concrete operator
space, then the mapping $$x \mapsto \begin{bmatrix} 0 & x \\ 0 & 0
\end{bmatrix}$$ into $B(H^2)$ is easily seen to be completely isometric. As abstract operator spaces, $V$ and its image under the above mapping are identical, even though $V$ could contain non-zero positive operators while its image in $B(H^2)$ does not. In other words, abstract operator spaces forget their order structure.

In this paper, we will demonstrate an abstract characterization for
operator spaces in terms of matrix gauges (Theorem \ref{mainThm}).
Concretely, we define the gauge of an operator $T \in B(H)$ to be
$\nu(T)=\|Re(T)_+\|$, where $Re(T)_+$ is the positive part of the
real part of $T$. From the gauge, we can recover the norm (see Lemma
\ref{GaugeNormLemma}), involution and order structure at every
matrix level. Any mapping which is ``completely gauge isometric"
will automatically be completely isometric, self-adjoint, and a
complete order embedding. In the course of proving our main result,
we will also recover an abstract characterization for unital
operator spaces in terms of cones of \textbf{accretive operators},
operators whose real part is positive (Theorem \ref{UnitalOpSp}). As
applications, we will consider representations of normal operator
spaces and extensions of completely positive completely contractive
maps. A normal operator space is an abstract operator space together
with an order structure which satisfies the condition that $x \leq y
\leq z$ implies $\|y\| \leq \max(\|x\|,\|z\|)$ at every matrix
level. We provide representation theorems for these objects
(Theorems \ref{accRep} and \ref{normalRep}). This is achieved by
proving that the norm and order structures on a normal operator
space are induced by a matrix gauge. We show by example that this
inducing matrix gauge is not always unique. In fact, we can
characterize operator spaces with unique inducing gauges as operator
spaces with the ``real-cpcc extension property". This is is the
property that every completely positive completely contractive
linear map from the given operator space into $B(H)$ can be extended
to a completely positive completely contractive map on any
containing operator system. For example, operator systems have the
real-cpcc extension property, by the Arveson extension theorem.


Before moving on, we briefly review some related literature. The
results in sections 2 and 3 rely heavily upon the theory of
accretive operators and real-completely positive maps. Cones of
accretive operators have been studied in the context of operator
algebras, unital operator spaces, and Banach algebras in
\cite{BeardenBlecherRoots}, \cite{BlecherRealPosBanach},
\cite{BlecherOzawa}, \cite{BlecherReadI}, \cite{BlecherReadII}, and
\cite{BlecherRead}. See section 3 of \cite{BlecherReadII} for
several fundamental properties of the cone of accretive operators.
The real-completely positive maps defined in section 2 were also
studied by Blecher, Read and other authors. See section 2 of
\cite{BeardenBlecherRoots} for several fundamental results
concerning these maps. The study of accretive operators and
real-completely positive maps was brought to the author's attention
by David Blecher. Abstract characterizations of unital operator
spaces up to complete isometry can be found in \cite{BlecherNealI}
and \cite{BlecherNealII}. Another characterization of unital
operator spaces (in terms of the existence of sufficiently many
unital functionals) can be found in \cite{unitalNg}. Cones of
accretive operators are not addressed in \cite{BlecherNealI},
\cite{BlecherNealII} or \cite{unitalNg}. Matrix gauges, considered
in section 3, were introduced by Effros and Winkler in
\cite{Winkler} as non-commutative generalizations of Minkowski
gauges. Effros and Winkler were able to prove analogues of the
classical bipolar and Hahn-Banach Theorems for matrix gauges. We
prove a special case of their Hahn-Banach Theorem using our results
(Theorem \ref{EffrosWinkler}). Questions about abstract operator
spaces with a matricial order structure go back to the work of
Schreiner who studied ``matrix regular operator spaces" in
\cite{SCHREINER1998136}. The matrix regular condition is similar to,
but more restrictive than, our notion of normality. For example, the
positive cone of a matrix regular operator space spans the entire
space, while this may not be the case in a normal operator space.
Abstract characterizations of non-unital self-adjoint ordered
operator spaces can also be found in \cite{Karn10}, \cite{Ng}, and
\cite{Werner02}. Each of these authors take as an axiom the
existence of sufficiently many positive functionals to norm the
space, whereas we make no such assumption. Normality is also
mentioned in Werner's paper \cite{Werner02} in connection with the
existence of sufficiently many positive functionals to norm the
space. An abstract characterization for matrix-ordered $*$-algebras
due to Juschenko and Popovych can by found in
\cite{JuschenkoPopovychMatrixOrdAlg}.


We now summarize some basic definitions, notation, and background.
For a detailed introduction to these topics, we refer the reader to
\cite{paulsen2002completely}. We will call a vector space $V$ a
$\mathbb{R}$-vector space (respectively $\mathbb{C}$-vector space)
if the underlying field is $\mathbb{R}$ (respectively,
$\mathbb{C}$). For any subset $S$ of a $\mathbb{C}$-vector space, we
let ``span $S$" denote the set of $\mathbb{C}$-linear combinations
of elements of $S$. Let $V$ be a ($\mathbb{R}$ or
$\mathbb{C}$)-vector space. For each $n \in \mathbb{N}$, we let
$M_n(V)$ denote the vector space of $n \times n$ matrices with
entries in $V$. For each $n,m \in \mathbb{N}$, we let $M_n$
(respectively, $M_{n,m}$) denote the $n \times n$ (respectively, $n
\times m$) matrices with entries in $\mathbb{C}$. We use the
notation $x=[x_{k,l}] \in M_n(V)$ to indicate that the $(k,l)$-entry
of $x$ is $x_{k,l}$. When $x \in V$ and $Y = [y_{k,l}] \in M_n$, we
let $x \otimes Y$ denote the matrix $[y_{k,l} x] \in M_n(V)$. For $x
\in M_n(V)$, we let $x^T$ denote the transpose of $x$.

Given a Hilbert space $H$ and $T \in B(H)$, we let $\sigma(T)$
denote the spectrum of $T$ and, when $T$ is self-adjoint, we let
$T_+$ denote the positive part of $T$.

A subset $C$ of a ($\mathbb{R}$ or $\mathbb{C}$)-vector space is
called a cone if $C + C \subseteq C$ and $tC \subseteq C$ for all $t
\geq 0$. A cone is called \textbf{proper} if $C \cap -C = \{0\}$.
Now, let $V$ be a $\complex$-vector space. A \textbf{matrix cone} is
a sequence of cones $\{C_n \subset M_n(V)\}$ satisfying the
condition that for every $X \in M_{n,k}$, we have $X^*C_nX \subset
C_k$. If, in addition, each $C_n$ is proper, we call $\{C_n\}$ a
\textbf{proper} matrix cone.

By a \textbf{$*$-vector space}, we mean a $\complex$-vector space
$V$ equipped with a conjugate-linear \textbf{involution} $*$, i.e.,
a map satisfying $(x + \lambda y)^* = x^* + \overline{\lambda}y^*$
and $x^{**}=x$ for each $x, y \in V$ and $\lambda \in \complex$. For
each $x \in V$, we define $Re(x):= \frac{1}{2}(x + x^*)$ and $Im(x)
:= \frac{1}{2i}(x - x^*)$. We say $x$ is \textbf{self-adjoint} if
$x^*=x$, and write $V_{sa}$ for the $\reals$-vector space $\{x:
x^*=x\}$. When $V$ is a $*$-vector space, we may extend $*$ to
$M_n(V)$ by setting $[a_{k,l}]^* = [a_{k,l}^*]^T$, the transpose of
the matrix obtained by applying $*$ to each entry.

By a \textbf{matrix-ordered $*$-vector space}, we mean a $*$-vector
space $V$ together with a proper matrix cone $V_+ =
\{V_+^n\}$ satisfying $V_+^n \subset M_n(V)_{sa}$ for each $n$. We call a vector $e \in V_{sa}$ a \textbf{matrix-order unit} if for each $x \in M_n(V)_{sa}$, there exists a $t > 0$ such that $te \otimes I_n + x \in V_+^n$. We call $e$ an \textbf{archimedean matrix-order unit} if whenever $t e \otimes I_n + x \in V_+^n$ for all $t > 0$ it follows that $x \in V_+^n$. When $e$ is an archimedean matrix-order unit for a matrix-ordered $*$-vector space $(V,V_+)$, we call the triple $(V,V_+,e)$ an \textbf{archimedean matrix-order unit space}.

\begin{example} Let $H$ be a Hilbert space, and let $S \subset B(H)$
be an operator system. We may identify $M_n(S) \subset B(H^n)$. Let
$S_+^n$ be the set of positive operators in $M_n(S)$, $S_+ = \{S_+^n\}$, and $e =
I_H$. Then $(S,S_+,e)$ is an archimedean matrix-order unit space. \end{example}

Given two matrix-ordered $*$-vector spaces $(S,S_+)$ and $(T,T_+)$,
we call a map $\phi:S \rightarrow T$ \textbf{completely positive} if
the map $\phi^{(n)}: M_n(S) \rightarrow M_n(T)$ (defined by applying
$\phi$ to each entry of a matrix) is positive for each $n$ (i.e.,
$\phi^{(n)}(S_+^n) \subset T_+^n$). $\phi$ is called
\textbf{self-adjoint} if $\phi(x)^* = \phi(x^*)$ for each $x \in S$.
$\phi$ is called a \textbf{complete order embedding} if $\phi$ is
completely positive and one-to-one, and $\phi^{-1}$ (restricted to
the image of $\phi$) is also completely positive. We call a
surjective complete order embedding a \textbf{complete order
isomorphism}. It is well known that completely positive maps on
operator systems are automatically self-adjoint.

The following Theorem, due to Choi and Effros, shows that
archimedean matrix-order unit spaces can always be realized as
operator systems.

\begin{theorem}[Choi-Effros, \cite{ChoiEffros}] \label{ChoiEffros}
Let $(S,S_+,e)$ be an archimedean matrix-order unit space. Then there exists a Hilbert
space $H$ and a unital complete order embedding $\phi: S \rightarrow B(H)$. \end{theorem}

\noindent Hence, we may identify the archimedean matrix-order unit
space $S$ with the operator system $\phi(S) \subseteq B(H)$.

We now turn our attention to operator spaces. Let $V$ be a
$\complex$-vector space. We call a sequence of norms $\{ \|\cdot\|_n
: M_n(V) \rightarrow [0,\infty) \}$ an \textbf{$L^\infty$ matrix-norm} provided
that for each $A \in M_n(V)$, $B
\in M_m(V)$, and $X,Y \in M_{n,k}$, the following conditions hold.
\begin{enumerate}
\item $\|Y^*AX\|_k \leq \|X\| \|Y\| \|A\|_n$, where $\|X\|$ and
$\|Y\|$ are the operator norms of $X$ and $Y$, respectively.
\item $\|A \oplus B \|_{n+m} = \max \{ \|A\|_n, \|B\|_m\}$, where
$A \oplus B = \begin{bmatrix} A & 0 \\ 0 & B \end{bmatrix}$.
\end{enumerate}
A vector space $V$ together with an $L^\infty$ matrix-norm $\{\|\cdot\|_n\}$ is
called an \textbf{$L^\infty$ matrix-normed space}.

\begin{example} Let $V \subset B(H)$ be an operator space.
Define $\|\cdot\|_n$ on $M_n(V)$ by identifying
$M_n(V) \subset B(H^n)$ and defining $\|\cdot\|_n$ to be the
operator norm. Then $(V,\{\|\cdot\|_n\})$ is an $L^\infty$ matrix-normed space. \end{example}

Given two $L^\infty$ matrix-normed spaces $(V,\{\|\cdot\|_n\})$, and
$(W, \{\|\cdot\|_n\})$, we call a map $\phi: V \rightarrow W$
\textbf{completely contractive} if $\phi^{(n)}$ is contractive for
each $n$. We call $\phi$ \textbf{completely isometric} if
$\phi^{(n)}$ is isometric for each $n$.

The next theorem shows that every $L^\infty$ matrix-normed space has
a representation as an operator space.

\begin{theorem}[Ruan, \cite{Ruan}] \label{Ruan} Let $(V,\{\|\cdot\|_n\})$ be an
$L^\infty$ matrix-normed space. Then there exists a Hilbert space
$H$ and a linear map $\phi:V \rightarrow B(H)$ which is completely
isometric. \end{theorem}

\noindent Hence, we may identify the $L^\infty$ matrix-normed space $V$ with the operator space $\phi(V) \subset B(H)$, at least with respect to its norm structure.

\section{Accretive matrix-ordered vector spaces}

We begin by studying the structure of accretive matrix-ordered
vector spaces, algebraic generalizations of operator spaces with
cones of accretive operators (see Example \ref{concreteAccSp}
below). The following lemmas will be useful.

\begin{lemma} \label{lem1} Let $\{C_n\}$ be a matrix cone in a $\complex$-vector space
$Z$, and let $V_n = C_n \cap -C_n$. Then
each $V_n$ is a $\reals$-vector space. Moreover,
 $$ V_n = \{ z \in M_n(Z): z + z^T, i(z - z^T) \in M_n(V_1)\}.$$ \end{lemma}

\begin{proof} Since each $V_n$ is a cone and satisfies $V_n = -V_n$, each $V_n$ is a $\reals$-vector space. Set $V'_n := \{ z \in M_n(Z): z + z^T, i(z - z^T) \in M_n(V_1)\}.$
To see that $V_n \subseteq V'_n$, let $z=[z_{k,l}] \in V_n$. We will show that for each $k,l \leq n$, $z_{k,l} + z_{l,k}, i(z_{k,l} - z_{l,k}) \in V_1$. For each $k \leq n$, let $e_k
\in M_{n,1}$ be the column matrix with a 1 in its $k$th entry and
zeroes elsewhere. Now, $$ z_{k,l} + z_{l,k} = (e_k + e_l)^* z (e_k + e_l) - e_k^* z e_k - e_l^* z e_l $$ and $$ i(z_{k,l} - z_{l,k}) = (e_k + ie_l)^* z (e_k + ie_l) - e_k^* z e_k - e_l^* z e_l. $$ Since $\{V_m\}$ is a matrix cone and $V_n$ is a $\reals$-vector space, we see that $$z_{k,l} + z_{l,k}, i(z_{k,l} - z_{l,k}) \in V_1.$$ To see that $V'_n \subseteq V_n$, assume that $z=[z_{k,l}] \in V'_n$. It suffices to show that $z + z^T, z - z^T \in V_n$. Let $a_{k,l} := z_{k,l} + z_{l,k}, b_{k,l} := i(z_{k,l} - z_{l,k})$, and $c_{k,l} := z_{k,l} - z_{l,k}$. Since $z \in V'_n$, we see that $a_{k,l}, b_{k,l} \in V_1$ for each $k,l \leq n$. Now, $$ a_{k,l} \otimes (E_{k,l} + E_{l,k}) = (e_k + e_l)a_{k,l}(e_k + e_l)^* - e_k a_{k,l} e_k^* - e_l a_{k,l} e_l^* $$ and $$ c_{k,l} \otimes (E_{k,l} - E_{l,k}) = (e_k + ie_l)b_{k,l}(e_k + ie_l)^* - e_k b_{k,l} e_k^* - e_l b_{k,l} e_l^* $$ where $E_{k,l}$ is the $n \times n$ matrix with a 1 in the $(k,l)$ entry and zeroes elsewhere. Hence $a_{k,l} \otimes (E_{k,l} + E_{l,k}), c_{k,l} \otimes (E_{k,l} + E_{l,k}) \in V_n$. But $$ z + z^T = \frac{1}{2} \sum_{k,l \leq n} a_{k,l} \otimes (E_{k,l} + E_{l,k}) $$ and $$ z - z^T = \frac{1}{2} \sum_{k,l \leq n} c_{k,l} \otimes (E_{k,l} - E_{l,k}).$$ So $z+z^T, z - z^T \in V_n$.\end{proof}

\begin{lemma} \label{lem2} Let $\{C_n\}$ be a matrix cone
in a $\complex$-vector space $Z$, and let
$V_n = C_n \cap -C_n$. Then $\span V_n = M_n(\span
V_1)$. \end{lemma}

\begin{proof} For $z \in M_n(Z)$, set $a := z + z^T$ and $b := i(z - z^T)$. Then $z = \frac{1}{2} a + \frac{1}{2i} b$. If $z \in V_n$, then by Lemma 2.1, $a,b \in M_n(V_1)$. Hence, $V_n \subset M_n( \span V_1)$. On the other hand, if $z \in M_n(V_1)$, it is easily checked that $a + a^T, i(a-a^T), b+b^T, i(b-b^T) \in M_n(V_1)$. Hence, $a,b \in V_n$. So $M_n(V_1) \subset \span V_n$. The statement follows.\end{proof}

\begin{lemma} \label{lem3} Let $\{C_n\}$ be a matrix cone
in a $\complex$-vector space $Z$, and set
$V_n = C_n \cap -C_n$ and $J_n = V_n \cap iV_n$. Then each
 $J_n$ is a $\complex$-vector space, and $J_n = M_n(J_1)$.
\end{lemma}

\begin{proof} Since each $J_n$ is a $\reals$-vector space and since $iJ_n =
J_n$, it follows that $J_n$ is a $\complex$-vector space. Since
$\{C_m \cap iC_m\}$ is a matrix cone and $$J_n = (C_n \cap
iC_n) \cap -(C_n \cap iC_n), $$ it follows that $\span J_n = M_n(
\span J_1)$ by Lemma \ref{lem2}. But $J_n$ is a $\complex$-vector
space, so $J_n = M_n(J_1)$. \end{proof}

\begin{definition}
Let $Z$ be a $\complex$-vector space. A matrix cone
$C=\{C_n\}$ in $Z$ is called \textbf{$\complex$-proper} if
$\cap_{k=0}^3 i^k C_1 = \{0\}$, and the pair $(Z,C)$ is called an
\textbf{accretive matrix-ordered vector space}. In this case, we
often set $Z_{ac}^n = C_n$ and $Z_{ac}=C$. For each $n$, we set
$Z_{sa}^n := iZ_{ac}^n \cap -iZ_{ac}^n$ and $Z_+^n := Z_{sa}^n \cap
Z_{ac}^n$. We refer to the elements of $Z_{ac}^n$, $Z_{sa}^n$ and
$Z_+^n$ as the \textbf{accretive}, \textbf{self-adjoint} and
\textbf{positive} elements (respectively).
\end{definition}


By Lemma \ref{lem3}, we see that whenever $(Z,Z_{ac})$ is an accretive matrix-ordered vector space, $\cap_{k=0}^3 i^k Z_{ac}^n = \{0\}$ for each $n$. The following example motivates the above definition.

\begin{example} \label{concreteAccSp} Let $Z \subset B(H)$ be an operator space. Define $Z_{ac}^n$ to
be the cone of accretive operators, i.e., $Z_{ac}^n = \{ T \in
M_n(Z): Re(T) \geq 0 \}.$ Then $(Z,Z_{ac})$ is an accretive
matrix-ordered vector space. For, if $\pm Re(T), \pm Re(iT) \geq 0$,
then $Re(T)=Im(T)=0$, and hence, $T=0$. If $T \in Z_{sa}^n$, then $Im(T) = 0$, so $T=Re(T)$ is a self-adjoint operator. If $T \in Z_+^n$, then $T=Re(T) \geq 0$, so $T$ is a positive operator. \end{example}



The following lemma shows that the self-adjoint part of an accretive matrix-ordered vector space naturally possesses the structure of a matrix-ordered $*$-vector space (as defined in section 1).

\begin{lemma} \label{adjoint} Let $(Z,Z_{ac})$ be an accretive matrix-ordered vector space, and let $V = \span Z_{sa}^1$. Then for each $n$, $\span Z_{sa}^n = M_n(V)$ and there exists a unique
conjugate linear involution $*: M_n(V) \rightarrow M_n(V)$ such
that $z^*=z$ for each $z \in Z_{sa}^n$. Moreover, if $z = [z_{k,l}]
\in M_n(V)$, then $z^* = [z_{k,l}^*]^T$. If $z \in M_n(V)$, then
$z \in Z_{ac}^n$ if and only if $Re(z) \in Z_+^n$. Moreover, each $Z_+^n$ is proper. \end{lemma}

\begin{proof} By Lemma \ref{lem2}, $\span Z_{sa}^n = M_n(V)$. Now, suppose that $z \in M_n(V)$ and that $z = a + ib = c + id$ with $a,b,c,d \in Z_{sa}^n$. Since $a-c = i(d-b)$ and since $a-c, d-b \in Z_{sa}^n$, we see that $a-c \in Z_{sa}^n \cap iZ_{sa}^n = \{0\}$. Hence, $a=c$ and $d=b$. We conclude that each $z \in M_n(V)$ has a unique decomposition as $z = a + ib$ with $a,b \in Z_{sa}^n$. Define
$(a+ib)^* = (a-ib)$. Then $*$ is clearly the unique conjugate linear
involution on $M_n(V)$ such that $a^*=a$ when $a \in Z_{sa}^n$.

Next, assume that $z = [z_{k,l}] \in M_n(V)$ and that $z = a_1 +
ia_2$ with $a_j = [a_{k,l}^j] \in Z_{sa}^n$ for each $j$. Then the
$(k,l)$ entry of $z^*$ is $a_{k,l}^1 - ia_{k,l}^2$. Now, $$
a_{k,l}^j = \frac{1}{2}(a_{k,l}^j + a_{l,k}^j) + \frac{1}{2i}i
(a_{k,l}^j - a_{l,k}^j)$$ for each $j$. By applying Lemma \ref{lem1}
to $\{Z_{sa}^m\}$, we see that $a_{k,l}^j + a_{l,k}^j, i(a_{k,l}^j -
a_{k,l}^j) \in Z_{sa}^1$ for each $j$. Hence, \begin{eqnarray}
(a_{k,l}^j)^* & = & (\frac{1}{2}(a_{k,l}^j + a_{l,k}^j) +
\frac{1}{2i} i(a_{k,l}^j - a_{l,k}^j))^* \nonumber \\ & = &
\frac{1}{2}(a_{k,l}^j + a_{l,k}^j) - \frac{1}{2i} i(a_{k,l}^j -
a_{l,k}^j) = a_{l,k}^j \nonumber \end{eqnarray} for each $j$. So $
z_{l,k}^* = (a_{l,k}^1 + i a_{l,k}^2)^* = a_{k,l}^1 - ia_{k,l}^2.$
Thus the $(k,l)$ entry of $z^*$ is $z_{l,k}^*$. Therefore $z =
[z_{k,l}^*]^T$.

To prove the final claims, let $z \in M_n(V)$. Then $z = Re(z) +
iIm(z)$ with $Re(z), Im(z) \in Z_{sa}^n$. Because $\pm iIm(z) \in
Z_{ac}^n$, we see that $Re(z) \in Z_{ac}^n$ if and only if $z =
Re(z) + iIm(z) \in Z_{ac}^n$. Finally, if $\pm z \in Z_+^n$, then
$\pm iz, \pm z \in Z_{ac}^n$, and hence $z=0$ (since $Z_{ac}^n$ is
$\complex$-proper). So each $Z_+^n$ is proper. \end{proof}



Assume that $(Z,Z_{ac})$ and $(W,W_{ac})$ are two accretive
matrix-ordered vector spaces. We call a map $\phi:Z \rightarrow W$
\textbf{real-completely positive} if for each $n \in \naturals$,
$\phi^{(n)}(Z_{ac}^n) \subset W_{ac}^n$. If $\phi$ is one-to-one,
and if both $\phi$ and $\phi^{-1}$ (restricted to the range of
$\phi$) are real-completely positive, then we call $\phi$ a
\textbf{real-complete order embedding}. A surjective real-complete
order embedding is called a \textbf{real-complete order
isomorphism}.

We call an accretive matrix-ordered vector space
\textbf{self-adjoint} if $\span Z_{sa}^1 = Z$. We conclude this
section by showing that every accretive matrix-ordered vector space
is contained in a ``smallest" self-adjoint accretive matrix-ordered
vector space.

\begin{definition} Let $(Z,Z_{ac})$ be an accretive matrix-ordered vector space. We call a pair $((V,V_{ac}),\phi:Z \rightarrow V)$ a $\mathbf{*}$\textbf{-closure} of $(Z,Z_{ac})$ provided that $(V,V_{ac})$ is a self-adjoint accretive matrix-ordered vector space, $\phi$ is a real-complete order embedding, and $V = \phi(Z) + \phi(Z)^*$ as a $*$-vector space with respect to the unique involution $*$ on $V$ (provided in Lemma \ref{adjoint}). \end{definition}

\begin{lemma} \label{starClosureLem} Suppose that $((V,V_{ac}),\phi)$ is a $*$-closure for $(Z,Z_{ac})$.
Then $\phi^{(n)}(x) + \phi^{(n)}(y)^* \in V_{ac}^n$ if and only if
$x + y \in Z_{ac}^n$. Moreover, $\phi^{(n)}(x) + \phi^{(n)}(y)^* =
0$ if and only if $x,y \in \span Z_{sa}^n$ and $y = -x^*$.
\end{lemma}

\begin{proof} First, notice that $\phi^{(n)}(x) + \phi^{(n)}(y)^* \in
V_{ac}^n$ if and only if $Re(\phi^{(n)}(x) + \phi^{(n)}(y)^*) \in
V_+^n$ by Lemma \ref{adjoint}. But $Re(\phi^{(n)}(x) +
\phi^{(n)}(y)^*) = Re(\phi^{(n)}(x) + \phi^{(n)}(y))$. So
$\phi^{(n)}(x) + \phi^{(n)}(y)^* \in V_{ac}^n$ if and only if $x + y
\in Z_{ac}^n$, since $\phi$ is a real-complete order embedding. Now,
assume that $\phi(x) + \phi(y)^* = 0$. Since $i(\phi(x) + \phi(y)^*)
= \phi(ix) - \phi(iy)^*$, this implies that $\pm (x+y), \pm i(x-y)
\in Z_{ac}^n$. Hence, $(x-y),i(x+y) \in Z_{sa}^n$. Since $x =
\frac{1}{2}(x - y) + \frac{1}{2i}i(x + y)$, we see that $x \in \span
Z_{sa}^n$ and $y = -x^*$. On the other hand, if $y = -x^*$, then
$\pm Re(x+y) = \pm Re(i(x-y)) = 0$, and hence $\pm (\phi^{(n)}(x) +
\phi^{(n)}(y)), \pm i(\phi^{(n)}(x) + \phi^{(n)}(y)^*) = 0$. So
$\phi^{(n)}(x) + \phi^{(n)}(y)^* = 0$, since $V_{ac}^n$ is
$\complex$-proper. \end{proof}

The following theorem shows that every accretive matrix-ordered vector space has a unique $*$-closure.

\begin{theorem} \label{adjointify} Let $(Z,Z_{ac})$ be an accretive matrix-ordered
vector space. Then there exists a $*$-closure $((V,V_{ac}),\phi)$.
Moreover, if $((V',V'_{ac}),\psi)$ is some other $*$-closure of
$(Z,Z_{ac})$, then there exists a real-complete order isomorphism
$j:V \rightarrow V'$.
\end{theorem}

\begin{proof} Define $Z \times Z^*$ to be the vector space $\{(x,y):
x,y \in Z\}$ with entry-wise addition and scalar multiplication
defined by $\lambda(x,y) = (\lambda x, \overline{\lambda} y)$. Let
$C_n = \{(x,y) \in Z \times Z^* : x+y \in Z_{ac}^n\}$. Then
$C=\{C_n\}$ is a matrix cone in $Z \times Z^*$. Let $J =
\cap_{k=0}^3 i^k C_1$. By Lemma \ref{lem3}, $J$ is a
$\complex$-subspace of $Z \times Z^*$, and $M_n(J) = \cap_{k=0}^3
i^k C_n$.

Let $V = (Z \times Z^*) / J$ and identify $M_n(V)$ with $M_n(Z
\times Z^*) / M_n(J)$. Set $V_{ac}^n = \{z + M_n(J) : z \in C_n\}$. Then
$V_{ac}=\{V_{ac}^n\}$ is a matrix cone in $V$. Moreover,
$V_{ac}$ is $\complex$-proper, since $ \cap_{k=0}^3 i^k V_{ac}^n =
\{M_n(J)\}$. So $(V,V_{ac})$ is an accretive matrix-ordered vector
space.

We'll write $[(x,y)]$ for the coset $(x,y) + M_n(J)$. To see that
$V$ is self-adjoint, let $[(x,y)] \in M_n(V)$. Then \begin{equation} \label{cosetDecomp} [(x,y)] =
\frac{1}{2}[(x,y) + (y,x)] + \frac{1}{2i}i[(x,y)-(y,x)]. \end{equation} Now,
$$\pm i[(x,y)+(y,x)] = \pm [(ix+iy,-iy-ix)] \in V_{ac}^1$$ and $$\pm
[(x,y) - (y,x) ] = \pm [(x-y,y-x)] \in V_{ac}^1,$$ since $(ix + iy)
+ (-iy -ix) = 0 \in Z_{ac}^1$ and $(x-y) + (y-x) = 0 \in Z_{ac}^1$.
So $$[(x,y)+(y,x)], i[(x,y)-(y,x)] \in V_{sa}.$$ Hence, $V = \span
V_{sa}^1$.

Define $\phi: Z \rightarrow V$ by $\phi(z) = [(z,0)]$. Note that
$\phi(z) = [(0,0)]$ implies that $(z,0) \in J$, i.e., $\pm z,
\pm iz \in Z_{ac}^n$. But this only occurs when $z = 0$. So $\phi$
is one-to-one. Moreover, $[(z,0)] \in V_{ac}^n$ if and only if $z =
z + 0 \in Z_{ac}^n$, so $\phi$ is a real-complete order embedding.
Also, equation (\ref{cosetDecomp}) above implies that
$\phi^{(n)}(z)^* = [(0,z)]$ and hence $[(x,y)] = \phi^{(n)}(x) +
\phi^{(n)}(y)^*$ for each $x,y \in Z$. So $V = \phi(Z) + \phi(Z)^*$.

Finally, if $(V',\psi)$ is some other $*$-closure, define $j:V
\rightarrow V'$ by $j(\phi(x) + \phi(y)^*) = \psi(x) + \psi(y)^*$.
By Lemma \ref{starClosureLem}, $\psi(x) + \psi(y)^* = 0$ if and only
if $y=-x^*$ in $Z$. But this occurs if and only if $\phi(x) +
\phi(y)^* = 0$. Hence, $j$ is well-defined. Also, $j$ is a
real-complete order isomorphism, since $\phi^{(n)}(x) +
\phi^{(n)}(y)^* \in V_{ac}^n$ if and only if $x + y \in Z_{ac}^n$,
which occurs if and only if $\psi^{(n)}(x) + \psi^{(n)}(y)^* \in
(V')_{ac}^{n}$.
\end{proof}


\section{Unital Operator Spaces}

In this section, we briefly demonstrate an abstract characterization
for unital operator spaces, i.e., subspaces of $B(H)$ containing the
identity operator $I_H$. We regard such subspaces abstractly as
accretive matrix-ordered vector spaces with a specified unit $e$. We
begin by axiomatizing the role played by the unit with respect to
the accretive order structure.

Let $(Z,Z_{ac})$ be an accretive matrix-ordered vector space.
We call $e \in Z_{sa}^1$ an \textbf{accretive matrix-order unit} if for each $z \in M_n(Z)$ there is a $t > 0$ such that $te \otimes I_n + z \in Z_{ac}^n$. We call $e$ an \textbf{accretive archimedean matrix-order unit} if $te \otimes I_n + z \in Z_{ac}^n$ for all $t > 0$ implies that $z \in Z_{ac}^n$. In this case, we call the triple $(Z,Z_{ac},e)$ an \textbf{accretive archimedean matrix-order unit space}.


\begin{proposition} \label{opSys} Let $(Z,Z_{ac},e)$ be an accretive archimedean matrix-order unit space, and let $V = \span Z_{sa}^1$ and $V_+ = \{Z_+^n\}$. Then $(V,V_+,e)$ is an archimedean matrix-order unit space. \end{proposition}

\begin{proof} By Lemma \ref{adjoint}, $(V,V_+)$ is a matrix-ordered $*$-vector space. Since $e$ is self-adjoint and $x \in V_+^n$ if and only if $x \in Z_{ac}^n$ and $x$ is self-adjoint, it is easy to verify that $e$ is an archimedean matrix-order unit for $(V,V_+)$. \end{proof}

\begin{corollary} \label{SAUnitalRep} Let $(V,V_{ac},e)$ be a self-adjoint accretive archimedean matrix-order unit
space. Then there exists a Hilbert space $H$ and a real-complete
order embedding $\phi: V \rightarrow B(H)$.
\end{corollary}

\begin{proof} By Lemma \ref{adjoint} and Proposition \ref{opSys}, $V$ is an archimedean matrix-order unit space. By Theorem \ref{ChoiEffros}, there exists $\phi: V \rightarrow B(H)$ which is a self-adjoint complete order embedding. Since $z \in V_{ac}^n$ if and only if $Re(z) \in V_+^n$ by Lemma \ref{adjoint}, it is clear that $\phi$ is a real-complete order embedding. \end{proof}

We now consider the non-self-adjoint case.

\begin{proposition} \label{PropUnitalStarClosure} Let $(Z,Z_{ac},e)$ be an accretive archimedean matrix-order unit
space, let $((V,V_{ac}), \psi)$ be the $*$-closure of $Z$. Then
$f=\psi(e)$ is an accretive archimedean matrix-order unit for
$(V,V_{ac})$. \end{proposition}

\begin{proof}
Writing the elements of $V$ as $\phi(x) + \phi(y)^*$ for $x,y \in
Z$, we have $\phi^{(n)}(x) + \phi^{(n)}(y)^* \in V_{ac}^n$ if and
only if $x + y \in Z_{ac}^n$ by Lemma \ref{starClosureLem}. For
arbitrary $x,y \in M_n(Z)$, we may choose $t \geq 0$ such that $te
\otimes I_n + x + y \in Z_{ac}^n$, and hence $tf \otimes I_n +
\phi^{(n)}(x) + \phi^{(n)}(y)^* \in V_{ac}^n$. So $f$ is an
accretive matrix-order unit. If $tf\otimes I_n + \phi^{(n)}(x) +
\phi^{(n)}(y)^* \in V_{ac}^n$ for all $t > 0$, then $te \otimes I_n
+ x + y \in Z_{ac}^n$ for all $t > 0$ and hence $x + y \in
Z_{ac}^n$. Consequently, $\phi^{(n)}(x) + \phi^{(n)}(y)^* \in
V_{ac}^n$. So $f$ is archimedean. \end{proof}

We now state the main result of this section.

\begin{theorem} \label{UnitalOpSp} Let $(Z,Z_{ac}^n,e)$ be an accretive archimedean matrix-order unit
space. Then there exists a Hilbert space $H$ and a unital real-complete order embedding $\phi:Z \rightarrow B(H)$. Moreover, $\phi$ is completely isometric with respect to the norm $$ \|z\|_n = \inf \{ t > 0 : Re(
\begin{bmatrix} 0 & 2z \\
0 & 0 \end{bmatrix}) \leq te \otimes I_{2n} \}. $$ \end{theorem}

\begin{proof} The map $\phi$ is obtained by composing the maps from Corollary \ref{SAUnitalRep} and Proposition \ref{PropUnitalStarClosure}. The observation that the map $\phi$ in Theorem \ref{UnitalOpSp} satisfies
$\|z\|_n = \|\phi^{(n)}(z)\|_{B(H^n)}$ follows from Lemma 3.1 of \cite{paulsen2002completely} (see also Proposition 13.3 of
\cite{paulsen2002completely}). \end{proof}


\section{Matrix gauge spaces}

In this section, we will arrive at the main result of this paper
(Theorem \ref{mainThm}). Our goal here is to characterize the
(possibly non-unital) subspaces of $B(H)$ as matrix gauge spaces. Let $Z$ be a $\complex$-vector space. We call a sequence of
functions $\nu=\{\nu_n: M_n(Z) \rightarrow [0,\infty)\}$ a
\textbf{matrix-compatible function} provided that for each
$A \in M_n(Z), B \in M_m(Z)$, and each
scalar matrix $X \in M_{n,k}$, the following conditions hold.
\begin{enumerate} \item $\nu_k(X^*AX) \leq \|X\|^2 \nu_n(A)$. \item
$\nu_{n+m}(A \oplus B) = \max \{ \nu_n(A), \nu_m(B) \}$.
\end{enumerate}

\begin{lemma} \label{coneFromGauge} Let $\nu$ be a matrix-compatible
function on a $\complex$-vector space $Z$. For each $n \in
\naturals$, let $$C_n = \{z \in M_n(V): \nu_n(z)=0 \}, J_n = C_n
\cap -C_n \cap iC_n \cap -iC_n.$$ Then $\{C_n\}$ is a matrix
cone in $Z$, each $J_n$ is a $\complex$-vector space, and $J_n =
M_n(J_1)$. \end{lemma}

\begin{proof} Let $x,y \in C_n$. Since $x+y =
\begin{bmatrix} I_n & I_n \end{bmatrix} (x \oplus y) \begin{bmatrix} I_n & I_n
\end{bmatrix}^*$, we have \[ \nu_n(x + y) \leq \| \begin{bmatrix} I_n & I_n
\end{bmatrix} \|^2 \max\{ \nu_n(x), \nu_n(y)\} = 0. \] So $x+y \in C_n$. Also, for each $t
\geq 0$, $\nu_n(tx) = \nu_n((t^{1/2}I_n)x(t^{1/2}I_n)) \leq t
\nu_n(x) =0.$ So $tx \in C_n$. Hence, $C_n$ is a cone. If $X \in
M_{n,k}$, then for each $A \in C_n$ we have $\nu_k(X^*AX) \leq
\|X\|^2 \nu_n(A) = 0.$ Hence, $X^*AX \in C_k$. So $\{C_m\}$ is a
matrix cone. The statements concerning $J_n$ follow from Lemma \ref{lem3}. \end{proof}

Let $Z$ be a $\complex$-vector space. We call a map $\nu:Z
\rightarrow [0,\infty)$ a \textbf{gauge} provided that for each $x,y
\in Z$ and $t \geq 0$ the following hold.
\begin{enumerate}
\item $\nu(x+y) \leq \nu(x) + \nu(y)$. \item $\nu(tx)=t\nu(x)$.
\end{enumerate} If $\nu=\{\nu_n\}$ is a matrix-compatible function and
each $\nu_n$ is a gauge, we call $\nu$ a \textbf{matrix gauge}, and
we call the pair $(Z,\nu)$ a \textbf{matrix gauge space}.

Matrix gauges were defined by Effros and Winkler in \cite{Winkler},
where they were regarded as a generalizations of Minkowski gauges
for matrix-convex sets in $Z$. When the corresponding matrix-convex
set fails to be absorbing, the gauges may take on the value
$\infty$. However, this will not be the case for the gauges we
consider.

We call a gauge $\nu$ on a vector space $Z$
$\complex$\textbf{-proper} if whenever $z \in Z$ and $\nu(i^kz)=0$
for each $k \in \{0,1,2,3\}$, we have $z = 0$. We call a matrix
gauge $\nu$ $\complex$\textbf{-proper} if $\nu_1$ is
$\complex$-proper. By Lemma \ref{coneFromGauge}, each $\nu_n$ is
$\complex$-proper whenever $\nu_1$ is $\complex$-proper.

\begin{example} \label{concrete} Let $Z \subset B(H)$ be an operator space. For each $T \in M_n(Z)$, define $$\nu_n(T) = \|(Re(T))_+\|.$$ Then $\nu=\{\nu_n\}$ is a $\complex$-proper matrix gauge on $Z$. \end{example}

We will prove in Lemma \ref{GaugeNormLemma} that the gauge in Example
\ref{concrete} is a $\complex$-proper matrix gauge. We will then show in Theorem \ref{mainThm} that
whenever $(Z,\nu)$ is a matrix gauge space with a $\complex$-proper
gauge there exists a representation of $Z$ as a subspace of
$B(H)$, as in Example \ref{concrete}. Our strategy will be to show that each matrix gauge space embeds into a unital operator space.

\begin{definition} \label{DefUnitize} Let $(Z,\nu)$ be a matrix gauge space with
$\complex$-proper gauge. Define $\tilde{Z}=Z \times \complex$ with
entry-wise addition and scalar multiplication. Identifying
$M_n(\tilde{Z})$ with $M_n(Z) \times M_n$, define for each $(A,X)
\in M_n(\tilde{Z})$
$$ u_n(A,X) = \inf \{ t > 0 : X_t \gg 0,
\nu_n(X_t^{-1/2}AX_{t}^{-1/2}) \leq 1 \}$$ where $X_t := tI_n -
Re(X)$, and $X_t \gg 0$ means that $\sigma(X_t) \subset (0,\infty)$.
We call $(\tilde{Z}, u)$ the \textbf{unitization} of $(Z,\nu)$.
\end{definition}

In the next three lemmas, we will show that the matrix gauge $u$ in Definition \ref{DefUnitize} is a $\complex$-proper matrix gauge. We will write $(A,X) \oplus (B,Y)$ to mean $(A \oplus
B, X \oplus Y)$ and $T^*(A,X)T$ to mean $(T^*AT, T^*XT)$ for scalar
matrices $X,Y,T$ and non-scalar matrices $A$ and $B$.

\begin{lemma} \label{unitLemma} Let $(Z,\nu)$ be a matrix gauge space and $(\tilde{Z},u)$ its unitization. Then each $u_n$ is well-defined and the family $\{u_m\}$ is a
matrix-compatible function on $\tilde{Z}$. \end{lemma}

\begin{proof} For each scalar matrix $X \in M_n$
 there is a $t > 0$ such that $X_t \gg 0$,
since $\sigma(X_t) = t - \sigma(ReX)$. Also, when $X_t \gg 0$,
$\|X_t^{-1}\|$ is $\lambda^{-1}$, where $\lambda$ the smallest eigenvalue of $X_t$.
Thus, when $A \in M_n(Z)$ and $\nu_n(A) \neq 0$ we can choose $t>0$
large enough that $\|X_t^{-1}\| \leq \nu_n(A)^{-1}$, and thus
$\nu_n(X_t^{-1/2} A X_t^{-1/2}) \leq \|X_t^{-1}\| \nu_n(A) \leq 1$.
If $\nu_n(A) = 0$, then $\nu_n(X_t^{-1/2} A X_t^{-1/2}) \leq \|X_t^{-1}\| \nu_n(A) = 0$ for all $t$. It follows that $u_n(A,X)$ is well defined.

Now, let $A \in M_n(Z), B \in M_k(Z), X \in M_n$, and $Y \in M_k$.
To see that $u_{n+k}(A \oplus B, X \oplus Y) = \max\{u_n(A,X),
u_k(B,Y)\}$, notice that $(X \oplus Y)_t \gg 0$ if and only if $X_t
\gg 0$ and $Y_t \gg 0$, since $\sigma((X \oplus Y)_t) = \sigma(X_t)
\cup \sigma(Y_t)$. Also, $$ \nu_{n+k}( (X \oplus Y)_t^{-1/2} (A
\oplus B) (X \oplus Y)_t^{-1/2})$$ is equal to $$ \max\{
\nu_n(X_t^{-1/2} A X_t^{-1/2}), \nu_k(Y_t^{-1/2} B Y_t^{-1/2}) \} $$
since
\[ (X \oplus Y)_t^{-1/2} (A \oplus B) (X \oplus Y)_t^{-1/2} =
X_t^{-1/2} A X_t^{-1/2} \oplus Y_t^{-1/2} B Y_t^{-1/2}. \] Hence,
$$\nu_{n+k}( (X \oplus Y)_t^{-1/2} (A \oplus B) (X \oplus
Y)_t^{-1/2}) \leq 1$$ if and only if $\nu_n(X_t^{-1/2} A X_t^{-1/2})
\leq 1$ and $\nu_k(Y_t^{-1/2} B Y_t^{-1/2}) \leq 1$.

Finally, we show that $u_k(T^*(A,X)T) \leq \|T\|^2 u_n(A,X)$
whenever $T \in M_{n,k}$ and $A \in M_n(Z), X \in M_n$. Set $B =
T^*AT$ and $Y=T^*XT$. To prove the claim, it is enough to show that
for each $t
> u_n(A,X)$, $Y_r \gg 0$ and $\nu_k(Y_r^{-1/2}BY_r^{-1/2}) \leq 1$,
where $r = \|T\|^2 t$. That is, for each $t
> u_n(A,X)$, $\|T\|^2 t > u_k(B,Y)$. Choose $t>0$ such that $X_t \gg 0$ and $\nu_n(X_t^{-1/2} A
X_t^{-1/2}) \leq 1$. Then $Y_r \geq 0$, since
\[ Y_r = \|T\|^2 t I_n - Re(T^*XT) \geq t T^*T - T^*Re(X)T = T^*(X_t)T. \] Since
this holds for all $t > u_n(A,X)$, we may assume that $Y_r \gg 0$.
Let $W = X_t^{1/2} T Y_r^{-1/2}$. Then $Y_r^{-1/2} B Y_r^{-1/2} =
W^* X_t^{-1/2}AX_t^{-1/2} W$.  Also,
\begin{eqnarray} \|W\|^2 & = & \|W^*W\| \nonumber \\ & = & \|
Y_r^{-1/2} (T^* X_t T)Y_r^{-1/2} \| \nonumber \\ & \leq & \|
Y_r^{-1/2} Y_r Y_r^{-1/2} \| = 1 \nonumber
\end{eqnarray} since $0 \leq T^*X_tT \leq Y_r$, as shown above.
Consequently, \begin{eqnarray}
\nu_k(Y_r^{-1/2} B Y_r^{-1/2}) & = & \nu_k( W^*
X_t^{-1/2}AX_t^{-1/2} W) \nonumber \\ & \leq & \|W\|^2
\nu_n(X_t^{-1/2} A X_t^{-1/2}) \nonumber \\ & \leq & 1. \nonumber
\end{eqnarray} This proves the final claim. \end{proof}

\begin{lemma} \label{archimedean order unitgauge} Let $(Z,Z_{ac},e)$ be an accretive
archimedean matrix-order unit space. For each $z \in M_n(Z)$, define
$\nu_e^n(z) := \inf \{ t > 0 : te \otimes I_n - z \in Z_{ac}^n \}$
and set $\nu_e=\{\nu_e^n\}$. Then $(Z,\nu_e)$ is a matrix gauge
space. Moreover, for each $n \in \naturals$, $$Z_{ac}^n = \{z \in
M_n(Z): \nu_{e,n}(-z) = 0\}.$$
\end{lemma}

\begin{proof} By Theorem \ref{UnitalOpSp}, we may identify $Z$ with
a unital subspace of $B(H)$ for some Hilbert space $H$. Under this
identification, we have $$\nu_e^n(z) = \inf \{t > 0: Re(z) \leq t I_H \otimes I_n \}.$$ Fix $A \in M_n(Z)$, $B \in M_m(Z)$, and $X \in M_{n,k}$. Since $Re(A \oplus B) = Re(A) \oplus Re(B)$, we see that $Re(A \oplus B) \leq t I_H \otimes I_{n+m}$ if and only if $Re(A) \leq t I_H \otimes I_n$ and $Re(B) \leq t I_H \otimes I_m$. Also, if $Re(A) \leq t I_H \otimes I_n$ then $Re(X^*AX) \leq t\|X^2\| I_H \otimes I_k$. For, if $h \in H^k$ with $\|h\|^2=1$, then $$ \langle Re(X^*AX) h, h \rangle = \langle Re(A) Xh, Xh \rangle \leq \|Xh\|^2 t \leq \|X\|^2 t. $$ So $\{\nu_e^n\}$ is a matrix-compatible function. It is easily verified that each $\nu_e^n$ is a gauge and that $z \in M_n(Z)$ if and only if $\nu_e^n(-z)=0$. Since $Z_{ac}$ is $\complex$-proper, it follows that $\{\nu_e^n\}$ is $\complex$-proper, completing the proof.\end{proof}

Let $(Z,\nu)$ and $(W,\o)$ be two matrix gauge spaces. We call a map
$\phi:Z \rightarrow W$ \textbf{completely gauge contractive} if for
each $z \in M_n(Z)$ we have
$\o_n(\phi^{(n)}(z)) \leq \nu_n(z)$. We call $\phi$
\textbf{completely gauge isometric} if for each
$z \in M_n(Z)$ we have $\o_n(\phi^{(n)}(z)) = \nu_n(z)$.

\begin{lemma} \label{opSysRep} Let $(Z,\nu)$ be a
matrix gauge space and $(\tilde{Z},u)$ its unitization. Set
$\tilde{Z}_{ac}^n = \{(A,X) : u_n(-A,-X) = 0\}$. Then $(\tilde{Z},
\tilde{Z}_{ac}, e)$ is an accretive archimedean matrix-order unit
space for $e=(0,1)$, and \[ u_n(A,X) = \inf \{t>0 : te \otimes I_n -
(A,X) \in \tilde{Z}_{ac}^n \}
\] for each $A \in M_n(Z)$, $X \in M_n$.
Consequently, $(\tilde{Z},u)$ is a matrix gauge space. Moreover, the
mapping $z \mapsto (z,0)$ from $Z$ to $\tilde{Z}$ is completely
gauge isometric. \end{lemma}

\begin{proof} We first show that $u_n(A,0) = \nu_n(A)$ for each $A
\in M_n(Z)$. As $tI_n = 0_t \gg 0$ for every $t > 0$, we have
$u_n(A,0) = \inf \{ t > 0 : \nu_n(t^{-1}A) \leq 1 \} = \nu_n(A).$

Set $C_n = \{(A,X) \in M_n(\tilde{Z}): u_n(A,X)=0\}$. By Lemma
\ref{coneFromGauge}, $C=\{C_n\}$ is a matrix cone. To see
that $C$ is $\complex$-proper, it suffices to show that $u_1$ is
$\complex$-proper, by Lemma \ref{coneFromGauge} . To this end, let
$z \in Z, \lambda \in \complex$, and assume that $u_1(i^k z,i^k
\lambda)=0$ for each $k \in \{0,1,2,3\}$. Then for every $t > 0$,
$(i^k \lambda)_t = t - Re(i^k \lambda) > 0$. Hence, $Re(\lambda),
Im(\lambda) = 0$, so $\lambda = 0$. Since $u_1(z,0)=\nu_1(z)$ and
since $\nu$ is $\complex$-proper, we see that $z = 0$. So $u_1$ is
$\complex$-proper, and hence $C$ is a $\complex$-proper
matrix cone. Now set $\tilde{Z}_{ac}^n = -C_n$ for each $n
\in \naturals$. Then $\tilde{Z}_{ac}= \{ \tilde{Z}_{ac}^n\}$ is a
$\complex$-proper matrix cone. We now show that
\begin{equation} \label{unitizationGaugeEq} u_n(A,X) = \inf \{t>0 : t(0,I_n) - (A,X) \in
\tilde{Z}_{ac}^n \}. \end{equation} Now, $t(0,I_n) - (A,X) \in
\tilde{Z}_{ac}^n$ if and only if $u_n(A,X - tI_n) = 0$. This holds
if and only if there is a sequence $r_k \downarrow 0$ such that $(X
- tI_n)_{r_k} = (t + r_k)I_n - Re(X) \gg 0$, and \[ \nu_n((X -
tI_n)_{r_k}^{-1/2} A (X - tI_n)_{r_k}^{-1/2}) \leq 1 \] for all $k$.
Equivalently, by setting $s_k = r_k + t$, we see that $s_k
\downarrow t$, $X_{s_k} \gg 0$, and \[ \nu_n((X)_{s_k}^{-1/2} A
(X)_{s_k}^{-1/2}) \leq 1 \] for all $k$. Hence, $t(0,I_n) - (A,X)
\in \tilde{Z}_{ac}^n$ if and only if $u_n(A,X) \leq t$. The claim
follows. From Equation (\ref{unitizationGaugeEq}) above, it is easy
to check that $(0,1)$ is an accretive archimedean matrix-order unit
for $\tilde{Z}$, completing the proof. \end{proof}

Before stating our main result, we prove one final lemma.

\begin{lemma} \label{GaugeNormLemma} Let $H$ be a Hilbert space and let $Z$ be an operator space in $B(H)$. Then for each $z \in M_n(Z)$, $\inf \{t > 0: tI_H \otimes I_n
- z \in Z_{ac}^n \} = \|Re(z)_+\|_{B(H^n)}$ and $\|z\|_{B(H^n)} = \|Re(y)_+\|_{B(H^{2n})}$ for $$ y = \begin{bmatrix}
0 & 2z \\ 0 & 0
\end{bmatrix} \in B(H^{2n}).$$ \end{lemma}

\begin{proof}
Let
$a = Re(z)$, and identify $a$ with the function $f_a: x \mapsto x$
on $\sigma(a) \subset \reals$ in the unital $C^*$-algebra generated
by $a$. Then $a_+$ is identified with the function $(f_a)_+$ which
equals $f_a$ when $x \geq 0$ and equals zero otherwise. Identifying $I_H \otimes I_n$ with the function $I(x)=1$, we see that $\|(f_a)_+\|_{\infty} \leq t$ if and only if $(f_a)_+ \leq tI$. The final observation is now immediate from Theorem \ref{UnitalOpSp}.
\end{proof}

We now arrive at our main result, which follows immediately from Theorem \ref{UnitalOpSp}, Lemma \ref{opSysRep}, and Lemma \ref{GaugeNormLemma}. We regard $B(H)$ as a matrix gauge
space by equipping it with the matrix gauge defined in Example
\ref{concrete}.

\begin{theorem} \label{mainThm} Let $(Z,\nu)$ be a matrix gauge space with a
$\complex$-proper gauge. Then there exists a Hilbert space $H$ and a
completely gauge isometric linear map $\phi:Z \rightarrow B(H)$.
Moreover, $\phi$ is completely isometric with respect to the norm
defined by $$ \|z\|_n = \nu_{2n}(\begin{bmatrix} 0 & 2z
\\ 0 & 0 \end{bmatrix}) $$ and a real-complete order embedding with respect to the accretive cones defined by setting $Z_{ac}^n = \{z : \nu_n(-z) = 0\}$. \end{theorem}


\section{Representations of ordered operator spaces}

In this section, we consider the problem of representing an abstract
operator space (in the sense of Theorem \ref{Ruan}) equipped with a
compatible order structure as a subspace of $B(H)$ for some Hilbert
space $H$. We will demonstrate a simple compatibility condition that
is required for such representations to exist.

We begin with the non-self-adjoint case. Let $Z$ be a
$\complex$-vector space. We call a matrix gauge $h=\{h_n\}$ a
\textbf{hermitian matrix gauge} if $h$ is $\complex$-proper and if
for each $z \in M_n(Z)$ and $t \in \reals$, we have
$h_n(tz)=|t|h_n(z)$. In other words, a hermitian matrix gauge is a
$\complex$-proper gauge for which $h_n$ is a $\reals$-seminorm on
$M_n(Z)$ (i.e., a seminorm on the $\reals$-vector space $M_n(Z)$
obtained by forgetting the $\complex$-vector space structure).

Let $(Z,Z_{ac})$ be an accretive matrix-ordered vector space with a
hermitian matrix gauge $h$. We call $Z_{ac}$ \textbf{$h$-closed} if
whenever $\{z_k\}$ is a sequence in $Z_{ac}^n$ and $h_n(z_k - z)
\rightarrow 0$ for some $z \in M_n(Z)$, we have $z \in Z_{ac}^n$. We
call $(Z,Z_{ac},h)$ a \textbf{normal accretive operator space} if
$Z_{ac}$ is $h$-closed and if whenever $y-x,z-y \in Z_{ac}^n$, we
have $h_n(y) \leq \max\{h_n(x),h_n(z)\}.$

\begin{example} \label{normAccRep} Let $Z \subset B(H)$ be an operator space and let
$Z_{ac}$ be the matrix cone of accretive operators. For each
and $z \in M_n(Z)$, define $h_n(z)=\|Re(z)\|$.
Then $(Z,Z_{ac}^n,h)$ is  normal accretive operator space. For,
$Z_{ac}$ is easily seen to be $h$-closed, and if $y-x,z-y \in
Z_{ac}^n$, then $Re(x) \leq Re(y) \leq Re(z)$, and hence, $\|Re(y)\|
\leq \max\{ \|Re(x)\|, \|Re(z)\| \}$.
\end{example}

We will show below that every normal accretive operator space has a
representation as in Example \ref{normAccRep}. Our strategy will be
to show that normal accretive operator spaces can be replaced with
matrix gauge spaces. We will then apply Theorem \ref{mainThm}. To
this end, we first show how to define the matrix gauge.

\begin{definition} Let $(Z,Z_{ac},h)$ be a normal accretive
operator space. For each $z \in M_n(Z)$, define $\nu_{max}^n(z) =
\inf \{ h_n(z+p) : p \in Z_{ac}^n \}.$ We call
$\nu_{max}=\{\nu_{max}^n\}$ the \textbf{maximal matrix gauge for
$Z$}.
\end{definition}

We say that a matrix gauge $\nu$ \textbf{induces} $(Z,Z_{ac},h)$, a
(necessarily) normal accretive operator space, if for each $z \in
M_n(Z)$,
$$Z_{ac}^n = \{z : \nu_n(-z)=0\}, h_n(z) = \max\{ \nu_n(z),
\nu_n(-z)\}.$$ We will see in Proposition \ref{maxGauge} that the
maximal matrix gauge for a normal accretive operator space $Z$ is
its largest possible inducing matrix gauge.

\begin{proposition} \label{accGauge} Let $(Z,Z_{ac},h)$ be a normal accretive
operator space. Then $\nu_{max}$ is a $\complex$-proper inducing matrix gauge on $(Z,Z_{ac},h)$. \end{proposition}

\begin{proof} Since $Z_{ac}$ is $h$-closed, it is clear from the
definition of $\nu_{max}$ that $\nu_{max}$ induces $Z_{ac}$. Also, since $0 \in Z_{ac}^n$, it is clear
that $\max\{ \nu_{max}^n(z), \nu_{max}^n(-z)\} \leq h_n(z).$ Now,
if $p,q \in Z_{ac}^n$ and $z \in M_n(Z)$, then $z - (z-q), (z+p) - z
\in Z_{ac}^n$. Hence, $h_n(z) \leq \max \{h_n(q - z), h_n(z + p)
\}$. By taking an infimum over all $p,q \in Z_{ac}^n$, we see that
$$h_n(z) \leq \max\{ \nu_{max}^n(z), \nu_{max}^n(-z)\}.$$ Finally,
if $\nu_{max}^n(i^k z) = 0$ for all $k \in \{0,1,2,3\}$, then $i^k z
\in Z_{ac}^n$ for each $k$. Since $Z_{ac}$ is $\complex$-proper,
$z=0$ in this case. Hence, $\nu_{max}$ is $\complex$-proper.
\end{proof}

\begin{theorem} \label{accRep} Let $(Z,Z_{ac},h)$ be a normal accretive operator
space. Then there exists a Hilbert space $H$ and a real-complete order embedding $\phi:Z
\rightarrow B(H)$ such that for each $z \in
M_n(Z)$, $h_n(z)=\|Re(\phi^{(n)}(z))\|.$ \end{theorem}

\begin{proof} Combine Proposition \ref{accGauge} and Theorem
\ref{mainThm} along with the observation that $\|Re(T)\| = \max\{\|Re(T)_+\|, \|Re(T)_-\|\}$ in any $C^*$-algebra. \end{proof}

We now consider the self-adjoint case. Let $V$ be a $*$-vector
space. We call $(V,\{ \|\cdot\|_n\})$ a \textbf{$*$-operator space}
if $(V,\{ \|\cdot\|_n\})$ is an $L^\infty$ matrix-normed space and
if $\|x^*\|_n = \|x\|_n$ for each $x \in M_n(V)$. If $(V,V_+)$ is a
matrix-ordered $*$-vector space, then $(V,V_+,\{\|\cdot\|_n\})$ is
called a \textbf{matrix-ordered $*$-operator space} if each cone
$V_+^n$ is closed with respect to $\|\cdot\|_n$. We call
$(V,V_+,\{\|\cdot\|_n\})$ \textbf{normal} if whenever $x,y,z \in
V_{sa}^n$ satisfy $x \leq y \leq z$, we have $\|y\|_n \leq \max\{
\|x\|_n, \|z\|_n\}$.

\begin{example} Let $V \subset B(H)$ be a self-adjoint subspace, and
let $V_+$ be the matrix-ordering of positive operators. Then $(V,
V_+, \{\|\cdot\|_n\})$ is a normal matrix-ordered $*$-operator
space. \end{example}

Before moving on, we briefly recall some history. Matrix-ordered
$*$-operator spaces were considered by Werner in \cite{Werner02}. It
was shown that to each matrix-ordered $*$-operator space, there
exists a complete order embedding into $B(H)$. In general, this
embedding need not be isometric. In fact, it was shown in
\cite{Blecher} that the operator space dual $A'$ of a $C^*$-algebra
$A$ of dimension at least 2 possesses the structure of a
matrix-ordered $*$-operator space - however, no order embedding of
$A'$ into $B(H)$ is isometric. This failure is due to the fact that
$A'$ is not normal (to see this, consider the Jordan decomposition
of a functional on a $C^*$-algebra). We now show that completely
isometric representations are possible when the space is normal.

\begin{theorem} \label{normalRep} Let $(V,V_+,\{\|\cdot\|_n\})$ be a normal
matrix-ordered $*$-operator space. Then there exists a Hilbert space
$H$ and a self-adjoint complete order embedding $\phi: V \rightarrow
B(H)$ which is completely isometric. \end{theorem}

\begin{proof} Define $V_{ac}$ by setting $V_{ac}^n = \{x \in
M_n(V) : x + x^* \in V_+^n \}$, and for each $x \in M_n(V)$, define
$h_n(x) = \|Re(x)\|_n$, where $Re(x)= \frac{1}{2}(x + x^*)$. We
leave it to the reader to verify that $(V, V_{ac}, h )$ is a normal
accretive operator space. By Theorem \ref{accRep}, there exists a
Hilbert space $H$ and a real-complete order embedding $\phi: V \rightarrow B(H)$ with
$h_n(x) =
\|Re(\phi^{(n)}(x))\|_n$. Since $$ \|x\|_n = \|\begin{bmatrix} 0 & x \\
x^* & 0 \end{bmatrix} \|_{2n} = h_{2n}( \begin{bmatrix} 0 & 2x \\ 0
& 0
\end{bmatrix} )$$ we see that $\phi$ is completely isometric. Since
every real-complete order embedding is automatically a self-adjoint complete order
embedding, we are done.
\end{proof}

Proposition \ref{accGauge} shows that a
normal accretive operator space $(Z,Z_{ac},h)$ is induced by its
maximal gauge $\nu_{max}$. The following examples show that
$\nu_{max}$ is not, in general, the only gauge which induces
$(Z,Z_{ac},h)$. Recall the gauge $\nu_e$ defined in Lemma
\ref{archimedean order unitgauge}. It follows easily from Lemma \ref{archimedean order unitgauge} that
$\nu_e$ induces any unital accretive operator space (as well as its
non-unital subspaces).

\begin{example} Let $V = \span \{x:=(-2,0,1)\} \subset \complex^3$, where $\complex^3$ is
regarded as the diagonal of $M_3 = B(\complex^3)$. Then
$\nu_{max}^1(x)=2$ whereas $\nu_e^1(x) = 1$. For, $$ \nu_{max}^1(x) = \inf \{ \|(-2,0,1) + p\| : p \in V_+^1 \} = 2$$ since $V_+^1 = \{0\}$, while $\nu_e^1(x) = \inf \{ t > 0 : (-2,0,1) \leq t(1,1,1) \} = 1.$ \end{example}

In the above example, the fact that $V_+^1 = \{0\}$ forced the
maximal gauge to be different from the order unit gauge. The next
example illustrates that even when $V_+^1$ spans $V$ the maximal
gauge may not be unique.

\begin{example} Fix $n > 2$. Let $V = \span \{ x:=(2,n,0), y:=(0,n,1) \}
\subset \complex^3$. Then $\nu_e^1(y-x) = 1$, whereas
$\nu_{max}^1(y-x) = \frac{2n}{n+2}$. To see this, notice that $p \in V_+$ if and only if $p = rx + sy$ for some $r,s \geq 0$. Set $m(r,s) = \max(|2r - 2|,|n(r+s)|,|1+s|)$. Then $\nu_{max}^1(y-x) = \inf \{ m(r,s): r,s \geq 0\}$. Clearly $m(r,s) \leq m(r,s')$ whenever $s \leq s'$, so we may assume $s=0$. Hence, we seek $r \geq 0$ minimizing $m(r,0) = \max(|2r - 2|, nr, 1)$. This is achieved at $r = \frac{2}{n+2}$. Since $n > 2$, we see that $\nu_{max}^1(y-x) = m(\frac{2}{n+2},0) = \frac{2n}{n+2}$. \end{example}

We will show in Theorem
\ref{BigcpccExt} that uniqueness of the inducing gauge is equivalent
to a fundamental extension property for operator spaces. While inducing gauges may not be unique, every
inducing gauge is bounded above by $\nu_{max}$.

\begin{proposition} \label{maxGauge} Let $(Z,Z_{ac},h)$ be a normal matrix-ordered
operator space, and suppose that $\nu$ is an inducing gauge for $Z$.
Then for each $z \in M_n(Z)$, we have $\nu_n(z) \leq
\nu_{max}^n(z)$. \end{proposition}

\begin{proof} Given $z \in M_n(Z)$ and $p \in Z_{ac}^n$, we have
that $$\nu_n(z) \leq \nu_n(z+p) + \nu_n(-p) = \nu_n(z+p) \leq
h_n(z+p).$$ The result follows by taking an infimum. \end{proof}

We conclude this section by demonstrating one situation in which we have a unique inducing gauge.

\begin{proposition} Let $(Z,Z_{ac},e)$ be an accretive archimedean matrix-order unit
space. For each $z \in M_n(Z)$, let $$h_n(z) = \inf \{t > 0 : z - t
e \otimes I_n, t e \otimes I_n - z \in Z_{ac}^n \}.$$ Suppose that
$\nu$ induces $(Z,Z_{ac},h)$. Then $\nu=\nu_e$. \end{proposition}

\begin{proof} Write $e_n = e \otimes I_n$. First, if $z \in Z_{ac}^n$, then $\nu_e^n(z) = h_n(z) =
\nu_n(z)$, and if $-z \in Z_{ac}^n$, then $\nu_e^n(z)=0=\nu_n(z)$.
So assume that neither $z$ nor $-z$ is in $Z_{ac}^n$. Let $z' =
\nu_e^n(-z)e_n + z$. Note that $z' \in Z_{ac}^n$. Hence,
$\nu_e^n(z')=\nu_n(z')$. By the definition of $\nu_e$, $\nu_e^n(z')
= \nu_e^n(-z) + \nu_e^n(z)$. Also, $ \nu_n(z') \leq \nu_e^n(-z) +
\nu_n(z).$ However, $\nu_n(z) \leq \nu_e^n(z)$. For, if $t > 0$ and
$t e_n - z \in Z_{ac}^n$, then $\nu_n(z) = \nu_n(z - te_n + te_n)
\leq \nu_n(te_n)=t$. It follows that $\nu_e^n(-z) + \nu_e^n(z) =
\nu_e^n(-z) + \nu_n(z)$, and hence, $\nu_n(z)=\nu_e^n(z)$.
\end{proof}


\section{Extension problems}

We conclude by considering the problem of extending certain types of
operator-valued maps. We first relate completely gauge contractive
maps to completely positive, completely contractive, and real-completely positive maps.

Let $(V,V_{ac},h)$, $(W,W_{ac},h)$ be normal accretive operator
spaces. We call a linear map $\phi: V \rightarrow W$
\textbf{real-completely contractive} if it is gauge contractive with
respect to the $h$ gauges. Such maps are automatically completely
contractive and self-adjoint, as $h_n(x) = \|Re(x)\|_n$ in any
representation (as in Theorem \ref{accRep}). A map which is both
real-completely contractive and real-completely positive is called
\textbf{real-cpcc}.

\begin{proposition} \label{cpccExt} Let $(V,\nu)$, $(W,\o)$ be matrix gauge spaces, and let $(V,
V_{ac}, h)$ and $(W, W_{ac}, h')$ be the corresponding induced
accretive operator spaces, respectively. If $\phi: V \rightarrow W$
is completely gauge contractive, then $\phi$ is real-cpcc. Moreover,
if $\phi$ is real-cpcc, then $\phi$ is completely gauge contractive
with respect to the gauges $\nu_{max}$ and $\o$.
\end{proposition}

\begin{proof} First, assume $\phi$ is completely gauge contractive.
If $x \in V_{ac}^n$ then $\nu_n(-x)=0$. Hence, $\o_n(-\phi^{(n)}(x))
\leq \nu_n(-x)=0$. So $\phi^{(n)}(x) \in W_{ac}^n$. So $\phi$ is
real-completely positive. Also, since $\max\{ \o_n(\phi^{(n)}(x)),
\o_n(-\phi^{(n)}(x)) \} \leq \max \{ \nu_n(x), \nu_n(-x)\}$, we see
that $\phi$ is real-completely contractive.

Now, assume that $\phi$ is real-cpcc. Then for each
$x \in M_n(V)$ and $p \in V_{ac}^n$, we have
$$\o_n(\phi^{(n)}(x)) \leq \o_n(\phi^{(n)}(x) +
\phi^{(n)}(p)) \leq h_n(x + p).$$ Taking an infimum over all $p \in
V_{ac}^n$ completes the proof.
\end{proof}

The following fundamental theorem was demonstrated by Arveson
\cite{Arveson69}.

\begin{theorem}[Arveson] \label{Arveson} Let $S \subset \tilde{S}$ be an
inclusion of operator systems, and let $\phi: S \rightarrow B(H)$ be
a completely positive map. Then there exists a completely positive extension $\tilde{\phi}: \tilde{S}
\rightarrow B(H)$. \end{theorem}

\noindent We will use the following fact, which follows from Theorem
2.6 of \cite{BeardenBlecherRoots}. We thank David Blecher for
providing this reference. To keep our paper self-contained, we
present a brief proof as a corollary of Theorem \ref{Arveson} above.

\begin{corollary}[Bearden, Blecher, Sharma \cite{BeardenBlecherRoots}] \label{caExt} Let $Z \subset \tilde{Z}$ be an
inclusion of unital operator spaces, and let $\phi: Z \rightarrow
B(H)$ be a real-completely positive map. Then there exists a
real-completely positive extension $\tilde{\phi}: \tilde{Z}
\rightarrow B(H)$. \end{corollary}

\begin{proof} By rescaling $\phi$, we may assume that $\phi$ is completely contractive (with respect to the norm described in Theorem \ref{UnitalOpSp}). By Proposition 1.2.8 of \cite{Arveson69}, $\phi$ extends to a
unique completely positive map on $Z + Z^*$. By Theorem \ref{Arveson}, we may further
extend $\phi$ to a completely positive map $\tilde{\phi}: \tilde{Z} + \tilde{Z}^*
\rightarrow B(H)$. The restriction of $\tilde{\phi}$ to $\tilde{Z}$ is the
desired extension. \end{proof}

The next theorem follows from a result of Effros and Winkler on the
extension of completely gauge contractive maps (Theorem 6.9,
\cite{Winkler}). When $V \subset V'$ is an inclusion of matrix gauge
spaces, we shall assume $V$ is endowed with the restriction of the
gauge $\nu$ on $V'$ to $V$. Since $B(H)$ is an operator system, we
may regard $B(H)$ as a matrix gauge space with the unique inducing
gauge $\nu_e$.

\begin{theorem} \label{EffrosWinkler} Let $(V',\nu)$
be a matrix gauge space with a $\complex$-proper gauge $\nu$ and a subspace $V$, and let $\phi: V
\rightarrow B(H)$ be a completely gauge contractive map. Then there
exists a completely gauge contractive extension $\tilde{\phi}:
V' \rightarrow B(H)$. \end{theorem}

Effros and Winkler's proof of Theorem 6.9 in \cite{Winkler} is
non-trivial, even with the assumption that $\nu$ is
$\complex$-proper and finite.
We would like to demonstrate a short proof of Theorem
\ref{EffrosWinkler} using our results. We first prove the following.

\begin{lemma} \label{extLemma} Let $(V,\nu)$ be a matrix gauge space
and $S$ a unital operator system with order unit $e$, regarded as a
matrix gauge space with the matrix gauge $\nu_e$. If $\phi:V
\rightarrow S$ is completely gauge contractive, then the unique
unital extension $\tilde{\phi}: \tilde{V} \rightarrow S$ is
real-completely positive, where $\tilde{V}$ is the unitization of
$V$. \end{lemma}

\begin{proof} Suppose that $(A,X) \in \tilde{V}_{ac}^n$, i.e., suppose that
$u_n(-A,-X) = 0$. By Definition \ref{DefUnitize}, there exists a
decreasing sequence $\{t_k > 0\}$ with $t_k \rightarrow 0$, $t_kI_n
+ Re(X) = (-X)_{t_k} \gg 0$ and $\nu_n(-(-X)_{t_k}^{-1/2} A
(-X)_{t_k}^{-1/2}) \leq 1$ for each $k$.

Set $B_k = -(-X)_{t_k}^{-1/2} A (-X)_{t_k}^{-1/2}$. Since $\phi$ is
completely gauge contractive, we have \begin{equation}
\label{RealIneqEq} Re(\phi^{(n)}(B_k)) \leq \nu_n(B_k)e \otimes I_n
\leq e \otimes I_n.\end{equation} Since $Re(\phi^{(n)}(B_k)) =
(-X)_{t_k}^{-1/2}Re(-A)(-X)_{t_k}^{-1/2}$, we may conjugate Equation
(\ref{RealIneqEq}) above by $(-X)_{t_k}^{1/2}$ to obtain $ -Re(A)
\leq (-X)_{t_k} \otimes e = e \otimes Re(X) + t_k e \otimes I_n.$
Hence, $ 0 \leq Re(A) + e \otimes Re(X) + t_k e \otimes I_n =
Re(\tilde{\phi}^{(n)}(A,X)) + t_k e \otimes I_n.$ It follows that
$Re(\tilde{\phi}^{(n)}(A,X)) \geq 0$. Since this holds for all
$(A,X) \in \tilde{V}_{ac}^n$, $\phi$ is real-completely positive.
\end{proof}

We now prove Theorem \ref{EffrosWinkler}.

\begin{proof} Assume $\phi:V \rightarrow B(H)$ is completely gauge
contractive. Write $\tilde{V}$ and $\tilde{V'}$ for the unitizations
of $V$ and $V'$, respectively. By Lemma \ref{extLemma}, the unital
extension of $\phi$ to $\tilde{V}$ is real-completely positive. By
Corollary \ref{caExt}, we may further extend $\phi$ to a
real-completely positive map $\tilde{\phi}: \tilde{V'} \rightarrow
B(H)$. The restriction of $\tilde{\phi}$ to $V'$ is a completely
gauge contractive extension of $\phi$.
\end{proof}

We now consider the extension of real-cpcc maps. We first
demonstrate by example that real-cpcc maps need not have completely
positive completely contractive extensions to their unitizations.
For simplicity, we consider examples involving self-adjoint operator
spaces.

\begin{example} Fix $n > 2$ and let $V = \span \{x:=(-n,0,1)\} \subset \complex^3$, where $\complex^3$ is
regarded as the diagonal of $M_3 = B(\complex^3)$. Define $\phi: V
\rightarrow \complex$ by setting $\phi(x) = n$. Then $\phi$ is
real-cpcc (trivially), but $\phi$ has no completely positive
completely contractive extension to $\complex^3$, as any such
extension $\tilde{\phi}$ must satisfy $\tilde{\phi}(1,1,1) \geq
\phi(x) = n$. \end{example}

For a non-trivial example, consider the following.

\begin{example} Fix $n > 2$, and let $V_n = \span \{
x_n:=(n,n,0), y_n:=(0,n,1) \} \subset \complex^3$. Then the
functional defined by $\phi_n(x_n) = 0$ and
$\phi_n(y_n)=\frac{n}{2}$ is real-cpcc, but has no completely
positive completely contractive extension to $\complex^3$. To see
that $\phi$ is real-cpcc, it suffices to show that $\phi$ is
positive and that $\phi$ is contractive on $(V_n)_{sa}$. It is clear
that $\phi$ is positive. If $\| tx_n + ry_n \| \leq 1$ for some $t,r
\in \reals$, then $|t| \leq \frac{1}{n}$ and $|t + r| \leq
\frac{1}{n}$. Hence, $|r| \leq |r + t| + |t| \leq \frac{1}{n} +
\frac{1}{n}$. Therefore \[
|\phi(tx_n + ry_n)| = \frac{n|r|}{2} \leq 1.
\] So $\phi$ is contractive. However, any extension
$\tilde{\phi}$ of $\phi$ satisfies $$\tilde{\phi}(1,1,1) \geq
\tilde{\phi}(y_n - x_n) = \frac{n}{2} > 1.$$
\end{example}

In fact, a modification of the above example shows that completely
positive maps need not have completely positive extensions in the
absence of a unit.

\begin{example} For each $n > 2$, let $V_n = \span \{ x_n=(n,n,0), y_n=(0,n,1) \} \subset
\complex^3$, and let $V = \bigoplus_{n>2} V_n \subset l^\infty$.
Define $\phi: V \rightarrow B(l^2)$ by $\phi = \oplus_{n>2} \phi_n$,
where the $\phi_n$'s are the functionals from the previous example.
Then $\phi$ has no completely positive extension to $l^\infty$. For, any positive extension $\tilde{\phi}$ of $\phi$ would have to satisfy $\tilde{\phi}(I) \geq \phi_n(y_n - x_n) \oplus 0 \geq \frac{n}{2} \oplus 0$ for each $n > 2$.
\end{example}

The following theorem links the above failures to the non-uniqueness
of inducing gauges. We say that a normal accretive operator space
$(Z,Z_{ac},h)$ has the \textbf{real-cpcc extension property} if for
each operator system $S$ and each completely isometric real-complete
order embedding $j:Z \rightarrow S$, every real-cpcc map $\phi:Z
\rightarrow B(K)$ has a completely positive completely contractive
extension $\tilde{\phi}: S \rightarrow B(K)$ (for each Hilbert space
$K$).

\begin{theorem} \label{BigcpccExt} An operator space $Z$ has the real-cpcc extension property if and only if $(Z,
Z_{ac}, h)$ has a unique inducing gauge, where $Z_{ac}$ is the
matrix cone of accretive operators and $h_n(x)=\|Re(x)\|_n$.
\end{theorem}

\begin{proof} Assume that $S$ is an operator system and that $j:Z \rightarrow S$ is a
completely isometric real-complete order embedding. Then $$Z_{ac}^n
= \{x \in M_n(Z) : Re(j^{(n)}(x)) \geq 0\}$$ and $$h_n(x) =
\|Re(j^{(n)}(x))\|_n.$$ Consequently, $\nu_e$ induces
$(Z,Z_{ac},h)$, where $\nu_e$ is the unique inducing gauge on $S$.
If $(Z,Z_{ac},h)$ has a unique inducing gauge, then $\nu_e^n(x) =
\nu_{max}^n(x)$ for all $x \in M_n(Z)$. Hence, every real-cpcc map
$\phi:Z \rightarrow B(K)$ has a completely positive completely
contractive extension to $S$, by Theorem 6.4 and Proposition 6.1.
Now, assume that $(Z,Z_{ac},h)$ has an inducing gauge $\nu \neq
\nu_{max}$. By Theorem \ref{mainThm}, there exist Hilbert spaces
$H,K$ and completely gauge isometric maps $\phi:(Z, \nu_{max})
\rightarrow B(K)$ and $j:(Z, \nu) \rightarrow B(H)$. Let $S$ be the
operator system generated by $j(Z)$ in $B(H)$. Note that $\phi$ is
real-cpcc by Proposition 6.1. However, any completely contractive
extension $\tilde{\phi}:S \rightarrow B(K)$ fails to be completely
positive. For, if $\tilde{\phi}$ is a completely contractive
extension of $\phi$ to $S$, then for each $x \in M_n(Z)$,
$$\nu_{I_K}^n(T_x) = \nu_{max}^n(x)$$ for $T_x = \tilde{\phi}^{(n)}(j^{(n)}(x))$. However, the assumption
that $\nu \neq \nu_{max}$ implies that there exists some $x \in
M_n(Z)$ such that $\nu_n(x) < \nu_{max}^n(x)$. Hence,
$\nu_{I_K}^n(T_x) > \nu_e^n(j^{(n)}(x)).$ So for
$\nu_e^n(j^{(n)}(x)) < t < \nu_{I_k}^n(T_x)$, we have $te
\otimes I_n - Re(j^{(n)}(x)) \geq 0$ in $S$, but
$t\tilde{\phi}(e) \otimes I_n - Re(T_x)$ is not positive in $B(K)$ since $tI_K \otimes I_n - Re(T_x)$ is not positive in $B(K)$ and $\tilde{\phi}(e)$ is a contraction. \end{proof}

We note that Arveson's extension theorem is a special case of the
above theorem, as it could be rephrased as ``every operator system
has the real-cpcc extension property".

\section{Final remarks and acknowledgments}

We conclude by noting that much more can be said concerning matrix
gauge representations of operator spaces. For example, objects such
as quotients or tensor products have not been addressed here. We
plan to study these concepts in a future paper. We refer the reader
to an earlier paper \cite{Russell1} for a few details concerning
quotients, as well as ``commutative" versions of several results
presented in this paper.

The author would like to thank the following individuals for helpful
comments and discussions related to this paper and its predecessor
\cite{Russell1}: David Blecher, Allan Donsig, Douglas Farenick,
Rupert Levene, Vern Paulsen, David Pitts, and Mark Rieffel, as well
as Mark Tomforde for making several fruitful conversations possible. The author also extends his gratitude to the Associate Editor and referees for their careful reading of this manuscript and their valuable suggestions.

\bibliographystyle{plain}
\bibliography{references_2}


\end{document}